\documentclass{amsart}
\usepackage{amsmath}
\usepackage{amsthm}
\usepackage{amssymb}

\theoremstyle{plain}
\newtheorem{theorem}{Theorem}[section]
\newtheorem{proposition}[theorem]{Proposition}
\newtheorem{corollary}[theorem]{Corollary}
\newtheorem{lemma}[theorem]{Lemma}

\theoremstyle{definition}
\newtheorem{definition}[theorem]{Definition}
\newtheorem{notation}[theorem]{Notation}

\theoremstyle{remark}
\newtheorem{remark}[theorem]{Remark}
\newtheorem{example}[theorem]{Example}

\newcommand{\bF}{\mathbb{F}}
\newcommand{\bN}{\mathbb{N}}

\newcommand{\bQ}{\mathbb{Q}}
\newcommand{\bR}{\mathbb{R}}
\newcommand{\bZ}{\mathbb{Z}}

\newcommand{\cO}{\mathcal{O}}

\newcommand{\cM}{\mathcal{M}}

\newcommand{\cF}{\mathcal{F}}
\newcommand{\cG}{\mathcal{G}}

\newcommand{\fm}{\mathfrak{m}}

\newcommand{\fl}{\mathfrak{l}}
\newcommand{\fp}{\mathfrak{p}}
\newcommand{\fP}{\mathfrak{P}}

\newcommand{\ab}{{\mathrm{ab}}}

\newcommand{\Aut}{{\mathrm{Aut}}}

\newcommand{\Gal}{{\mathrm{Gal}}}
\newcommand{\Inn}{{\mathrm{Inn}}}

\newcommand{\Stab}{{\mathrm{Stab}}}

\newcommand{\Trans}{{\mathrm{Tr}}}

\newcommand{\Spec}{{\mathrm{Spec}\,}}

\newcommand{\Clp}{{\mathrm{Cl}^+}}
\newcommand{\Char}{{\mathrm{char}\,}}
\newcommand{\Lnp}{{\mathrm{ln}^{(p)}}}
\newcommand{\Lnpe}{{\mathrm{ln}^{\fp}}}

\newcommand{\Ln}{{\mathrm{ln}}}

\title{Quandles associated to Galois covers of arithmetic schemes}

\author{Nobuyoshi Takahashi}
\address{
Department of Mathematics, Graduate School of Science, Hiroshima University, 
1-3-1 Kagamiyama, Higashi-Hiroshima, 739-8526 JAPAN}
\email{tkhsnbys@hiroshima-u.ac.jp}


\begin{document}

\maketitle

\begin{abstract}
Let $X$ be a normal, separated and integral scheme of finite type over $\bZ$ 
and $\cM$ a set of closed points of $X$. 
To a Galois cover $\tilde{X}$ of $X$ unramified over $\cM$, 
we associate a quandle 
whose underlying set consists of points of $\tilde{X}$  lying over $\cM$. 
As the limit of such quandles 
over all \'etale Galois covers and all \'etale abelian covers, 
we define topological quandles 
$Q(X, \cM)$ and $Q^\ab(X, \cM)$, respectively. 

Then we study the problem of reconstruction. 
Let $K$ be $\bQ$ or a quadratic field, 
$\cO_K$ its ring of integers, 
$X=\Spec \cO_K\setminus\{\fp\}$ 
the complement of a closed point 
such that  $\pi_1(X)^\ab$ is infinite, 
and $\cM$ a set of maximal ideals with density $1$. 
Using results from $p$-adic transcendental number theory, 
we show that $K$, $\fp$ and the projection $\cM\to\Spec \bZ$ 
can be recovered from the topological quandle $Q(X, \cM)$ or $Q^\ab(X, \cM)$. 
\end{abstract}

\section{Introduction}

A quandle is a set $Q$ endowed with a binary operation $\rhd$ 
which satisfies the following conditions. 
\begin{enumerate}
\item
$q\rhd q=q$. 
\item
$s_q: Q\to Q; r\mapsto q\rhd r$ is a bijection. 
\item
$q\rhd(r\rhd s)=(q\rhd r)\rhd(q\rhd s)$. 
\end{enumerate}

While there were some previous works on similar algebraic structures, 
it began to be widely studied after the works of 
Joyce(\cite{Joyce1982}) and Matveev(\cite{Matveev1982}) 
relating them to knots. 
To any knot they associated a quandle, called the knot quandle. 
They used it to study invariants of the knot 
and proved that a knot is determined up to equivalence by its knot quandle. 

If $K$ is a knot in $\bR^3$, 
its knot quandle is the set of homotopy classes of ``nooses'' or ``lollipops'' 
with some fixed base point whose boundaries go around $K$ exactly once. 
If $[a]$ and $[b]$ are such classes 
and $\partial a$ denotes the boundary of $a$ regarded as a loop, 
then $[a]\rhd [b]$ is defined as the class of the noose 
given by appending $\partial a$ to $b$. 
Roughly speaking, it is the class of $b$ 
transformed by the monodromy along $\partial a$. 

In number theory, it has been noticed that 
there are various analogous phenomena between prime numbers and knots
(see e.g. \cite{Morishita2012}). 
The purpose of this paper 
is to define a certain analogue of the knot quandle for arithmetic schemes 
and study how much information about an arithmetic scheme one can recover 
from the associated quandle. 
The latter can also be seen as a quandle analogue 
of anabelian geometry, 
i.e. the reconstruction of objects in arithmetic geometry 
from their arithmetic fundamental groups
(see \cite[Ch. XII]{NSW2008}, especially Question 12.3.4). 

Let $X$ be  an arithmetic scheme, 
which will mean a scheme of finite type over $\bZ$ in this paper. 
We assume that it is normal, separated and integral. 
One important example is $\Spec\bZ\setminus\{(p)\}$ 
where $p$ is a prime number. 
This can be seen as an analogue of the complement of a knot in $\bR^3$. 
Let us fix a set $\cM$ of closed points of $X$, 
which will be regarded as a collection of loops. 
Let $\tilde{X}$ be a Galois cover of $X$ unramified over $\cM$. 
We allow ``Galois covers of infinite degree'' 
to include in our consideration the maximal unramified cover $X^{\mathrm{ur}}$ 
and the maximal unramified abelian cover $X^{\mathrm{ur}, \mathrm{ab}}$. 

We define the set $Q(\tilde{X}/X, \cM)$ 
as the inverse image of $\cM$ in $\tilde{X}$ 
and endow it with a quandle operation. 
In the analogue between primes and knots, 
Frobenius maps correspond to monodromies around loops. 
Thus for elements $x$ and $y$ of $Q(\tilde{X}/X, \cM)$, 
denote by $s_x$ the Frobenius automorphism of $\tilde{X}$ 
associated to $x$, 
and define $x\rhd y$ to be $s_x(y)$. 
(See Remark \ref{rem_knot_quandle} 
for how this construction is related to knot quandles.) 
We also define a topology on $Q(\tilde{X}/X, \cM)$ 
and show that $Q(\tilde{X}/X, \cM)$ is a topological quandle. 

Now let $\tilde{X}$ be either $X^{\mathrm{ur}}$ or $X^{\mathrm{ur}, \mathrm{ab}}$. 
We study how much information about $X$ and $\cM$ 
one can recover from $Q=Q(\tilde{X}/X, \cM)$. 
Write $G$ for the group $\Aut(\tilde{X}/X)$. 
Note that the quandle $Q$ has some explicit information 
about points of $\tilde{X}$, 
which $G$ does not. 
On the other hand, 
it is not obvious when $G$ can be reconstructed from $Q$. 

We restrict to the case of an arithmetic curve, 
i.e. an arithmetic scheme $X=\Spec\cO_K\setminus S$ 
where $K$ is a number field and $S$ is a finite set of primes. 
It is easier to recover the information 
when we consider sufficiently many loops. 
Assuming that $\cM$ is of density $1$, 
the Galois group $G$, and hence $\cM=Q/G$ as a set, 
can be reconstructed from the topological quandle $Q$ 
(Corollary \ref{cor_recover_group_and_primes}). 
In particular, $Q(X^{\mathrm{ur}, \mathrm{ab}}/X, \cM)$ 
has a description in terms of $Q(X^{\mathrm{ur}}/X, \cM)$. 

We further specialize to the case where $K$ is rational or quadratic, 
$S$ consists of one prime 
and $\pi_1^{\mathrm{ab}}(X)$ is infinite. 
Then our main theorem shows that 
pretty much can be recovered from $Q(X^{\mathrm{ur}, \mathrm{ab}}/X, \cM)$. 

\begin{theorem}[Theorem \ref{thm_recover}]
Let each of $K$ and $K'$ be either $\bQ$ or a quadratic field, 
$\fp$ and $\fp'$ primes in $K$ and $K'$ respectively, 
and $\cM$ and $\cM'$ sets of primes with density $1$ in $K$ and $K'$, 
respectively. 

Write $X=\Spec\cO_K\setminus\{\fp\}$ 
and $Q=Q(X^{\mathrm{ur}, \mathrm{ab}}/X, \cM)$. 
Define $X'$ and $Q'$ in the same way. 
Assume that $\pi_1^{\mathrm{ab}}(X)$ or $\pi_1^{\mathrm{ab}}(X')$ 
is infinite. 

If $\varphi: Q\to Q'$ is an isomorphism of topological quandles, 
then there exists an isomorphism $\sigma: K\to K'$ which maps $\fp$ to $\fp'$, 
and $\varphi$ commutes with the natural maps to $\Spec\bZ$. 

Furthermore, 
except in the real quadratic case, 
$\sigma$ can be chosen 
so that $\varphi$ is compatible with the isomorphism 
$X\to X'$ induced by $\sigma$. 
\end{theorem}

The basic idea for our proof is as follows. 
The Galois group $\pi_1^{\mathrm{ab}}(X)$ is 
already recovered in Corollary \ref{cor_recover_group_and_primes}, 
and the quandle structure gives us 
a map $r: \cM\to \pi_1^{\mathrm{ab}}(X)$, 
which is essentially the global reciprocity map in class field theory. 
According to class field theory 
there is a natural homomorphism $\cO_{K_\fp}^\times\to \pi_1^{\mathrm{ab}}(X)$ 
compatible with $r$, with finite cokernel. 
If available, this could be used as a ruler to help recovering $\fl$ from $r(\fl)$. 
Unfortunately, our reconstruction of $\pi_1^{\mathrm{ab}}(X)$ from $Q$ 
does not directly provide us with this map. 
Instead we may think of $r(\cM)$ as a sort of unlabelled scale: 
We look at ``arrangements'' $\{r(\fl_1), \dots, r(\fl_k)\}\subset\pi_1^{\mathrm{ab}}(X)$ 
for $\fl_1, \dots, \fl_k\in\cM$. 
Applying results from $p$-adic transcendental number theory, 
the $p$-adic Six-Exponentials Theorem and its generalizations, 
to certain ratios formed from $r(\fl_i)$, 
we show that if an isomorphism $\pi_1^{\mathrm{ab}}(X)\to \pi_1^{\mathrm{ab}}(X')$ 
is compatible with a bijection $\cM\to\cM'$, 
then the conclusion of the theorem must hold. 

\begin{remark}
Let us compare our approaches and results 
with those of previous works. 

(1)
In Tamagawa's reconstruction(\cite{Tamagawa1997}) of affine curves over finite fields 
from their Galois groups, 
the decomposition groups of points are characterized 
by group-theoretical properties, 
and the function field is recovered from these groups. 
In our case, while the map $r$ gives us 
decomposition groups at many unramified primes 
together with natural generators, 
we do without recovering 
the decomposition groups at ramified points. 
All we do here is to prove that there is no unexpected isomorphism 
of associated quandles, 
with the help of transcendental number theory. 
As a drawback, we do not know how to directly reconstruct $K$ as a set 
at this moment.  

(2)
Unlike in the case of curves over fields, 
there seem to be few works on reconstruction of arithmetic curves. 
Ivanov \cite{Ivanov2014} studies the problem of reconstruction 
of an arithmetic curve from its arithmetic fundamental group 
under the assumption that at least $2$ rational primes are invertible in $\cO_{K, S}$ 
(or, equivalently, lie completely below $S$). 

(3)
Mochizuki \cite[Theorem 6.4]{Mochizuki2008} associates a certain category $\mathcal{C}$, 
called an arithmetic Frobenioid, 
to a number field $F$ and its Galois extension $\tilde{F}$, 
and reconstructs $F$ and $\tilde{F}$ from 
the associated arithmetic Frobenioid $\mathcal{C}$. 

We note an intriguing similarity in the proofs of this theorem and ours, 
although the author does not know exactly how they are related. 
The proof in \cite[Theorem 6.4]{Mochizuki2008} applies 
\emph{archimedean} (i.e. original) Six-Exponentials Theorem 
in order to recover residue characteristics of points, 
while we use $p$-adic Six-Exponentials Theorem. 
\end{remark}

\medbreak
The rest of this paper is organized as follows. 
In Section 2, we give a brief introduction to quandles 
and topological quandles. 
In Section 3, 
we associate a topological quandle to a Galois cover 
of an arithmetic scheme and a set of unbranched points. 
Then we restrict to the case of arithmetic curves  
and prove results on reconstruction in Section 4. 
Finally, we study the automorphism groups 
of quandles associated to Galois covers in Section 5. 

\section{Quandles}

Let us recall the definition of a quandle. 
Note that our convention is a little different from that of Joyce. 

\begin{definition}
A quandle is a set $Q$ equipped with a binary operation $\rhd$ 
satisfying the following axioms. 
\begin{enumerate}
\item
For any $q\in Q$, $q\rhd q=q$ holds. 
\item
For any $q, r\in Q$, 
there exists a unique element $r'\in Q$ such that $q\rhd r'=r$. 
\item
For any $q, r, s\in Q$, 
$q\rhd(r\rhd s)=(q\rhd r)\rhd(q\rhd s)$. 
\end{enumerate}
We define $s_q: Q\to Q$ by $s_q(r)=q\rhd r$, 
which is bijective by (2), 
and write $q\rhd^{-1}r$ for $s_q^{-1}(r)$. 
\end{definition}

\begin{definition}
We define the automorphism group of $Q$ as 
\[
\Aut(Q):=\{ f: Q\to Q \mid f(q\rhd r)=f(q)\rhd f(r)\hbox{ and $f$ is bijective}\}, 
\]
the inner automorphism group of $Q$ as 
\[
\Inn(Q):=\langle s_q \mid q\in Q\rangle 
\]
and the group of transvections as 
\[
\Trans(Q):=\{s_{q_1}^{e_1}\cdots s_{q_n}^{e_n}
\mid n\in\bN, q_1, \dots, q_n\in Q, \sum_{i=1}^n e_i=0\}. 
\]
Since $s_q$ is an automorphism of $Q$ by Axioms (2) and (3), 
$\Inn(Q)$ is in fact a subgroup of $\Aut(Q)$. 
\end{definition}

\begin{example}
If $G$ is a group, 
each of the operations $g\rhd h:=g^{-1}hg$ and $g\rhd h:=ghg^{-1}$ 
makes $G$ a quandle. 
Such a quandle is called a conjugation quandle. 
\end{example}

\begin{example}\label{ex_quandle_representation}
Let $G$ be a group, 
$\{z_\lambda\}_{\lambda\in\Lambda}$ 
a family of elements of $G$ parametrized by $\Lambda$ 
and $\{H_\lambda\}_{\lambda\in\Lambda}$ 
a family of subgroups of $G$ 
such that $z_\lambda\in H_\lambda$ and 
$H_\lambda$ is contained in the centralizer of $z_\lambda$. 

On the set $Q:=\coprod_{\lambda\in\Lambda} G/H_\lambda$, 
define the following binary operation. 
If $x$ and $y$ are elements of $G$ and 
$\lambda$ and $\mu$ are elements of $\Lambda$, 
then 
\[
xH_\lambda\rhd yH_\mu := x z_\lambda x^{-1}y H_\mu. 
\]
This operation makes $Q$ a quandle, 
which we denote by $Q(G, \{H_\lambda\}, \{z_\lambda\})$. 
\end{example}

\begin{definition}
An augmented quandle is a pair $(Q, G)$ of a set $Q$ and a group $G$ 
together with an action of $G$ on $Q$ and a map $\varepsilon: Q\to G$ 
(called the augmentation map) 
satisfying the following conditions. 
\begin{itemize}
\item
$\varepsilon(q)q=q$. 
\item
$\varepsilon(gq)=g\varepsilon(q)g^{-1}$. 
\end{itemize}
Then we define an operation $\rhd$ by $q\rhd r:=\varepsilon(q)r$, 
which is easily seen to satisfy the quandle axioms. 

If $(Q, G)$ and $(Q', G')$ are augmented quandles, 
a homomorphism of augmented quandles from $(Q, G)$ to $(Q', G')$ 
is a pair of maps $Q\to Q'$ and $G\to G'$ 
compatible with the actions and the augmentation maps. 
\end{definition}
\begin{remark}
(1)
It also follows 
that $G$ acts on $Q$ by quandle automorphisms 
and $\varepsilon$ is a quandle homomorphism 
if we regard $G$ as a conjugation quandle. 

(2)
It is easy to see that $(Q, \Aut(Q))$ and $(Q, \Inn(Q))$ are 
augmented quandles by the mapping $q\mapsto s_q$ 
and that the associated operations coincide with 
the original operation on $Q$. 
\end{remark}

Next we introduce notions of topological quandle 
and topological augmented quandle. 

\begin{definition}
A topological quandle is a topological space $Q$ endowed with 
a quandle operation $\rhd$ 
such that $Q\times Q\to Q; (q, r)\mapsto q\rhd r$ 
and $Q\times Q\to Q; (q, r)\mapsto q\rhd^{-1} r$ 
are continuous. 

A topological augmented quandle is a pair $(Q, G)$ 
of a topological space $Q$ and a topological group $G$ 
together with a continuous action of $G$ on $Q$ 
and a continuous map $\varepsilon: Q\to G$ 
satisfying $\varepsilon(q)q=q$ and $\varepsilon(gq)=g\varepsilon(q)g^{-1}$. 
Again, $Q$ is a topological quandle by the operation $q\rhd r:=\varepsilon(q)r$. 
\end{definition}

\begin{remark}
In the definition of topological quandles in 
\cite[Definition 2.1]{Rubinsztein2007}, 
it is only assumed that $(q, r)\mapsto q\rhd r$ is continuous 
and $s_q$ is a homeomorphism for any $q\in Q$. 
The author doesn't know if this definition 
is equivalent to ours. 
\end{remark}

\begin{proposition}\label{prop_quandle_representation}
Let $(Q, G)$ be an augmented quandle 
and let $\{q_\lambda\}_{\lambda\in G\backslash Q}$ 
be a complete system of representatives 
for the quotient set $G\backslash Q$. 
Then there is a natural isomorphism of 
$Q(G, \{\Stab_G(q_\lambda)\}, \{\varepsilon(q_\lambda)\})$ and $Q$. 

If $(Q, G)$ is a topological augmented quandle, 
$Q$ is Hausdorff, 
each $G$-orbit in $Q$ is open 
and $G$ is compact, 
then the above isomorphism is also an isomorphism 
of topological quandles. 
\end{proposition}
\begin{proof}
For discrete quandles, 
the proof is the same as that of \cite[Theorem 7.2]{Joyce1982}. 

In the continuous case, 
let $\pi: G\to G/\Stab_G(q_\lambda)$ be the natural map, 
define $a: G\to Gq_\lambda$ by $a(g)=gq_\lambda$ 
and write $\bar{a}: G/\Stab_G(q_\lambda)\to Gq_\lambda$ 
for the induced bijection. 
The continuity of $\bar{a}$ follows from that of $a$. 
If $K$ is a closed subset of $G/\Stab_G(q_\lambda)$, 
then $\pi^{-1}(K)$ is compact. 
Since $Q$ is Hausdorff, 
$\bar{a}(K)=a(\pi^{-1}(K))$ is closed. 
Thus $\bar{a}$ is a homeomorphism. 

Since $Gq_\lambda$ is open in $Q$, 
the natural map 
$Q(G, \{\Stab_G(q_\lambda)\}, \{\varepsilon(q_\lambda)\})\to Q$
is a homeomorphism. 
\end{proof}

\section{Quandles associated to Galois covers}

Let $X$ be a normal, separated and integral scheme 
of finite type over $\bZ$. 
Let $K=K(X)$ be the function field of $X$ 
and $L$ an algebraic extension of $K$, 
possibly of infinite degree. 

We denote by $X_L$ the normalization of $X$ in $L$. 
Explicitly, it can be described as follows. 
For an open affine subset $\Spec A$ of $X$, 
let $A_L$ be the integral closure of $A$ in $L$. 
If $f$ is an element of  $A$, 
the integral closure $(A_f)_L$ of $A_f$ in $L$ 
is equal to $(A_L)_f$, 
the localization of $A_L$ by $f$. 
This allows one to glue the affine schemes $\Spec A_L$ together 
into a scheme $X_L$. 

Let $X_L^0$ denote the set of closed points of $X_L$. 
This set can be described in terms of finite extensions. 
Let 
\[
\cF(L/K) :=\{M \mid \hbox{$M$ is an intermediate field of $L/K$ which is finite over $K$}\}
\]
and 
\[
\cG(L/K) :=\{M\in \cF(L/K) \mid \hbox{$M$ is a Galois extension of $K$}\}. 
\]
For a point $P$ of a scheme, 
let $\bF(P)$ denote the residue field at $P$. 

\begin{proposition}
(1)
If $L\supseteq M\supseteq K$ are algebraic field extensions, 
the natural map $\varphi_{L, M}: X_L\to X_M$ is surjective, 
and $\varphi_{L, M}^{-1}(X_M^0)=X_L^0$. 

(2)
A point $x\in X_L$ is closed if and only if 
$\bF(x)$ is algebraic over a finite field. 

(3)
There is a natural bijection 
between $X_L^0$ and $\varprojlim_{M\in\cF(L/K)} X_M^0$. 
If $L$ is a Galois extension of $K$, 
then $X_L^0$ is also in bijection with 
$\varprojlim_{M\in\cG(L/K)} X_M^0$. 
\end{proposition}
\begin{proof}
(1)
The surjectivity follows from the local description of $X_L$ above 
and Lying-Over Theorem. 
A point $x$ of $X_L$ is closed if and only if 
it is closed in any affine open neighborhood 
of the form $\Spec A_L$. 
The latter condition is equivalent to 
saying that $\varphi_{L, K}(x)$ is closed in $\Spec A$ 
by Lying-Over and Going-Up Theorem. 
Thus $x$ is closed in $X_L$ if and only if 
$\varphi_{L, K}(x)$ is closed in $X$, 
i.e. $\varphi_{L, K}^{-1}(X_K^0)=X_L^0$. 
Hence $\varphi_{L, M}^{-1}(X_M^0)=X_L^0$ holds. 

(2)
It is well-known that a point $P$ of 
a scheme of finite type over $\bZ$ is closed 
if and only if $\bF(P)$ is finite. 
Since $\bF(P)$ is the quotient field 
of a finitely generated $\bZ$-algebra, 
it is finite if and only if 
it is algebraic over a finite field. 

Let $x$ be a point of $X_L$. 
Then $\bF(x)$ is algebraic over $\bF(\varphi_{L, K}(x))$. 
If $x$ is closed, then $\varphi_{L, K}(x)$ is closed by (1) 
and $\bF(\varphi_{L, K}(x))$ is finite. 
Conversely, if $\bF(x)$ is algebraic over a finite field, 
then so is $\bF(\varphi_{L, K}(x))$, and 
$\varphi_{L, K}(x)$ is closed. 
By (1), $x$ is closed. 

(3)
Note that $x$ is closed in $X_L$ if and only if 
$x$ is closed in \emph{some} neighborhood of the form $\Spec A_L$ 
by (2). 
Thus it suffices to prove the assertion for $X=\Spec A$. 

One can think of $A_L$ as the union of $A_M$ for all $M\in\cF(L/K)$. 
If $\fm\subset A_L$ is a maximal ideal, 
then $(\fm\cap A_M)_{M\in\cF(L/K)}$ is an element of 
$\varprojlim_{M\in\cF(L/K)} X_M^0$. 
Thus we have a natural map 
$f: X_L^0\to\varprojlim_{M\in\cF(L/K)} X_M^0$. 
From $\fm = \bigcup_{M\in\cF(L/K)} (\fm\cap A_M)$ 
it follows that $f$ is injective. 

If $(\fm_M)_{M\in\cF(L/K)}\in \varprojlim_{M\in\cF(L/K)} X_M^0$, 
then $\fp:=\bigcup \fm_M$ is a prime ideal of $A_L$. 
Since $\fm_M\cap M'=\fm_{M'}$ holds for $M\supseteq M'$, 
we have $\fp\cap M=\fm_M$. 
From (1) we see that $\fp$ is maximal, i.e. $\fp\in X_L^0$, 
and $f(\fp)$ is equal to $(\fm_M)_{M\in\cF(L/K)}$. 
Thus $f$ is surjective. 

If $L/K$ is a Galois extension, 
then $L$ is the union of finite Galois extensions, 
so we may restrict to Galois extensions $M$. 
\end{proof}

In the rest of this section, 
let $\cM$ be a subset of $X^0$ 
and $L$ a Galois extension of $K$ 
such that the Galois cover $\tilde{X}=X_L$ of $X$ is 
unramified over $\cM$. 

\begin{definition}
We define $Q(\tilde{X}/X, \cM)$ to be 
the set of points of $\tilde{X}$ lying over $\cM$. 
Using the notation in the proposition above, 
$Q(\tilde{X}/X, \cM)=\varphi_{L, K}^{-1}(\cM)$. 

If $M$ and $M'$ are intermediate fields of $L/K$ 
which are Galois extensions of $K$ satisfying $M\supseteq M'$, 
let $\pi_{M, M'}: Q(X_M/X, \cM)\to Q(X_{M'}/X, \cM)$ 
denote the restriction of $\varphi_{M, M'}$. 

Let $X^{\mathrm{ur}}$ 
denote the maximal unramified cover of $X$ 
and $X^{\mathrm{ur}, \mathrm{ab}}$ 
the maximal abelian unramified cover of $X$. 
We define $Q(X, \cM)$ and $Q^{\mathrm{ab}}(X, \cM)$ 
as $Q(X^{\mathrm{ur}}/X, \cM)$ and $Q(X^{\mathrm{ur, ab}}/X, \cM)$, 
respectively. 
\end{definition}

\begin{proposition}
(1)
The map $\pi_{L, M}$ is surjective. 

(2)
$Q(\tilde{X}/X, \cM)=\varprojlim_{M\in\cG(L/K)} Q(X_M/X, \cM)$. 
\end{proposition}
\begin{proof}
This follows immediately 
from the previous proposition. 
\end{proof}

When $[L: K]<\infty$, 
we endow $Q(\tilde{X}/X, \cM)$ with the discrete topology. 
In the general case, we endow $Q(\tilde{X}/X, \cM)$ with the topology 
as a projective limit, 
i.e. the topology induced by the product topology 
on $\prod_{M\in\cG(L/K)} Q(X_M/X, \cM)$. 

Since $\pi_{L, K}^{-1}(a)$ is a compact open set for any $a\in \cM$, 
we see that $Q(\tilde{X}/X, \cM)$ is a locally compact space.

\begin{notation}
If $x$ is an element of $Q(\tilde{X}/X, \cM)$ 
we write $q(x)$ for $\#\bF(\pi_{L, K}(x))$ 
and $c(x)$ for $\Char\bF(x)$. 
We may identify $c$ with the natural map $Q(\tilde{X}/X, \cM)\to\Spec\bZ$. 

We denote by $s_x$ the Frobenius automorphism of $\tilde{X}$ associated to $x$, 
i.e. the unique element of $\Aut(\tilde{X}/X)$ such that 
$s_x(x)=x$ and $s_x^*(f)\equiv f^{q(x)}\mod \fm_x$ for any $f\in \cO_{\tilde{X}, x}$, 
where $\fm_x$ is the maximal ideal of $\cO_{\tilde{X}, x}$. 

We denote the corresponding element of $\Gal(L/K)$ by $F_x$. 
\end{notation}

\begin{definition}
Define the binary operation $\rhd$ on $Q(\tilde{X}/X, \cM)$ by
\[
x\rhd y := s_x(y). 
\]
\end{definition}

\begin{proposition}\label{prop_arithmetic_quandle}
(1)
The pair $(Q(\tilde{X}/X, \cM), \Aut(\tilde{X}/X))$, 
together with the natural action and the map 
\[
s_L: Q(\tilde{X}/X, \cM)\to \Aut(\tilde{X}/X);\ x\mapsto s_x, 
\]
is a topological augmented quandle 
with the associated quandle operation $\rhd$. 

(2)
For an intermediate Galois extension $M$ of $L/K$, 
the projection maps $\pi_{L, M}: Q(\tilde{X}/X, \cM)\to Q(X_M/X, \cM)$ 
and $\pi'_{L, M}: \Aut(\tilde{X}/X)\to\Aut(X_M/X)$
form a homomorphism of topological augmented quandles. 
\end{proposition}
\begin{proof}
We will prove the assertions for finite Galois extensions $L$ over $K$. 
Then we see assertion (1) in the general case by taking the projective limit. 
For (2), 
noting that $Q(X_M/X, \cM)=\varprojlim_{L'\in\cG(L/K)} Q(X_{L'\cap M}/X, \cM)$, 
we can deduce the assertion that $(\pi_{L, M}, \pi'_{L, M})$ 
is a homomorphism of topological augmented quandles 
from the case of finite extensions. 

So we assume that $L$ is a finite Galois extension of $K$. 
In this case all relevant spaces are endowed with discrete topology, 
so the continuity is obvious. 

For (1), 
by the definition of Frobenius automorphism, 
we have $s_x(x)=x$ for any $x\in Q(\tilde{X}/X, \cM)$. 

Let us show that $s_{g(x)}=g\circ s_x\circ g^{-1}$ holds 
for any $g\in \Aut(\tilde{X}/X)$ and $x\in Q(\tilde{X}/X, \cM)$. 
We set $x'=g(x)$ 
and prove that $g\circ s_x\circ g^{-1}$ 
satisfies the conditions required for $s_{x'}$. 
We have 
\begin{eqnarray*}
g\circ s_x\circ g^{-1}(x') & = & g(s_x(x)) \\
 & = & g(x) = x'. 
\end{eqnarray*}
For any $f\in \cO_{\tilde{X}, x'}$, 
we have 
$(g\circ s_x\circ g^{-1})^*f \in \cO_{\tilde{X}, x'}$ 
from the above, 
and the assumption on $s_x$ implies that 
\begin{eqnarray*}
(g\circ s_x\circ g^{-1})^*f & = & 
(g^*)^{-1}(s_x^*(g^*(f))) \\ 
& = & (g^*)^{-1}((g^*(f))^q+r),  
\end{eqnarray*}
where $r$ is an element of $\fm_x$ and $q=q(x)=q(x')$. 
The right hand side is equal to 
$f^q+(g^*)^{-1}(r)$, hence congruent to $f^q$ modulo $\fm_{x'}$. 
Thus $g\circ s_x\circ g^{-1}$ is the Frobenius automorphism associated to $x'$. 

For (2), 
we show that 
$F_{\pi_{L, M}(x)}=F_x|_M$ holds for any $x\in Q(\tilde{X}/X, \cM)$. 
Let $g$ be the automorphism of $X_M$ associated to $F_x|_M$. 
If $f\in\cO_{X_M, \pi_{L, M}(x)}$, 
then it can also be considered as an element of $\cO_{\tilde{X}, x}$ 
and therefore $F_x(f)$ is contained in $\cO_{\tilde{X}, x}\cap M=\cO_{X_M, \pi_{L, M}(x)}$. 
This shows that $g(\pi_{L, M}(x))=\pi_{L, M}(x)$. 

We have $F_x(f)-f^q\in\fm_x\cap \cO_{X_M, \pi_{L, M}(x)} = \fm_{\pi_{L, M}(x)}$, 
so $g$ is the Frobenius automorphism associated to $\pi_{L, M}(x)$ 
and we have $F_{\pi_{L, M}(x)}=F_x|_M$. 
This means $s_M\circ\pi_{L, M}=\pi'_{L, M}\circ s_L$. 
The actions are obviously compatible with the projection maps, 
so $(\pi_{L, M}, \pi'_{L, M})$ is a homomorphism of augmented quandles. 
\end{proof}

\begin{remark}\label{rem_knot_quandle}
If $K$ is a knot in $\bR^3$, 
its knot quandle $Q_K$ has the following description 
analogous to the arithmetic case. 

Let $\varphi: \tilde{X}\to \bR^3\setminus K$ be a universal cover, 
$U$ a tubular neighborhood of $K$ 
and $\Sigma$ the boundary of $U$. 
By Theorem 16.1 of \cite{Joyce1982}, 
the knot quandle $Q_K$ of $K$ can be identified with  
$\pi_0(\varphi^{-1}(\Sigma))$, 
the set of connected components of $\varphi^{-1}(\Sigma)$. 

The quandle operation is given in the following way. 
Let $\Sigma_1$ and $\Sigma_2$ be connected components of $\varphi^{-1}(\Sigma)$. 
Let $m_{\Sigma_1}$ be the deck transformation of $\Sigma_1\to\Sigma$ corresponding to a meridian 
(we don't have to care about the base point 
since the class of a meridian is in the center of $\pi_1(\Sigma)$, which is abelian). 
There is a unique lift of $m_{\Sigma_1}$ to a deck transformation of $\tilde{X}\to \bR^3\setminus K$, 
which we denote by $\tilde{m}_{\Sigma_1}$. 
Then $\Sigma_1\rhd \Sigma_2=\tilde{m}_{\Sigma_1}(\Sigma_2)$. 

If we stick to the analogy between loops and primes, 
$Q_K$ can be described as follows. 
Let $C$ be a meridian circle. 
Then we can define a quandle structure on $\pi_0(\varphi^{-1}(C))$. 
It is easy to prove that $Q_K$ is isomorphic to the quotient of $\pi_0(\varphi^{-1}(C))$
by the group generated by the class of a longitude. 
\end{remark}

\begin{proposition}\label{prop_galois_quandle_representation}
Write $Q=Q(\tilde{X}/X, \cM)$ and $G=\Aut(\tilde{X}/X)$ 
and let $\{q_a\}_{a\in\cM}$ be 
a complete system of representatives 
of $G\backslash Q$, which is naturally in bijection with $\cM$. 
Then $Q$ is naturally isomorphic to 
$Q(G, \{\overline{\langle s_{q_a}\rangle}\}, \{s_{q_a}\})$. 
\end{proposition}
\begin{proof}
This follows from 
Proposition \ref{prop_quandle_representation}, 
since the stabilizer of $q_a$ is $\overline{\langle s_{q_a}\rangle}$. 
\end{proof}

\begin{definition}
Let $X$ and $Y$ be topological spaces. 
We consider the set $C(X, Y)$ of continuous maps from $X$ to $Y$ 
with the compact-open topology. 
In other words, it has a subbasis consisting of the sets 
\[
O(C, U):=\{f\in C(X, Y) \mid f(C)\subseteq U\}
\]
for various compact subsets $C$ of $X$ and open subsets $U$ of $Y$. 
\end{definition}
If $X$, $Y$ and $Z$ are locally compact, 
it is known that $C(X, Y)$ is locally compact 
and that the evaluation map $C(X, Y)\times X\to Y$ 
and the composition map 
$C(Y, Z)\times C(X, Y)\to C(X, Z)$ are continuous. 

\begin{proposition}\label{prop_continuity_action_hom}
Let $Q$ be the quandle $Q(\tilde{X}/X, \cM)$. 
Then the natural map $\gamma: \Aut(\tilde{X}/X)\to C(Q, Q)$ is continuous. 
\end{proposition}
\begin{proof}
This follows from the continuity of the action 
and the following general fact: 
If $X, Y$ and $Z$ are locally compact spaces, 
a map $X\to C(Y, Z)$ is continuous 
if and only if the associated map $X\times Y\to Z$ is continuous. 
\end{proof}

\section{Reconstruction of arithmetic curves}

Now we restrict to the case of spectrums of integer rings 
and study the problem of reconstruction. 

Let $K$ be a number field, $\cO_K$ the ring of integers in $K$, 
and $S$ and $\cM$ disjoint sets of closed points of $\Spec\cO_K$ 
with $\# S<\infty$. 
We take $\Spec\cO_{K, S}=\Spec\cO_K\setminus S$ as $X$.

\subsection{Reconstruction of the Galois group}

Let $L$ be a Galois extension of $K$ which is unramified over $\cM$, 
and write $G$ for $\Aut(X_L/X)$ and 
$Q$ for $Q(X_L/X, \cM)$. 

\begin{proposition}
Assume that $\cM$ has density $1$. 

(1)
The natural map $\gamma: G\to C(Q, Q)$ is injective. 
Hence it is a homeomorphism onto its image. 

(2)
The closure of $\Inn(Q)$ in $C(Q, Q)$ is equal to $\gamma(G)$. 
\end{proposition}
\begin{proof}
(1)
Let $g\not= id_{X_L}$ be an element of $G$. 
Then there exists an intermediate finite Galois extension $M$ 
such that the induced automorphism $g_M$ of $M$ is not the identity. 
By Density Theorem, 
there exists an element $a\in \cM$ which decomposes completely in $M$. 
Then $\pi_{M, K}^{-1}(a)$ is a torsor over $\Aut(X_M/X)$, 
so the action of $g_M$ on this fiber is not trivial. 
Therefore $\gamma(g)$ is not the identity element. 

The map $\gamma$ is a homeomorphism onto its image 
since it is continuous 
by Proposition \ref{prop_continuity_action_hom}, 
injective as we just saw,  
$G$ is compact 
and $C(Q, Q)$ is Hausdorff. 

(2)
By the definition of the quandle operation, 
$\Inn(Q)$ is a subset of $\gamma(G)$. 
Let $g_0$ be any element of $G$. 
Then any neighborhood $V$ of $g_0$ 
contains a subset of the form $g_0\cdot\Aut(X_L/X_M)$ 
with $M$ a finite Galois extension of $K$. 
This is the set of elements of $G$ which induce 
the same element of $\Aut(X_M/X)$ as $g_0$. 
By Density Theorem, 
$\{s_y \mid y\in Q(X_M/X, \cM)\}=\Aut(X_M/X)$ holds. 
Thus one can find $x\in Q$ such that $s_x\in V$. 

It follows that $\Inn(Q)$ is dense in $\gamma(G)$. 
The latter is closed since $\gamma$ is continuous, 
$G$ is compact 
and $C(Q, Q)$ is Hausdorff. 
\end{proof}
\begin{remark}
As can be seen from the proof, 
$\{s_x \mid x\in Q\}$ is already dense in $\gamma(G)$. 
It follows that 
$\gamma(G)$ is also equal to the closure of $\Trans(Q)$. 
\end{remark}

\begin{corollary}\label{cor_recover_group_and_primes}
Assume that $\cM$ has density $1$. 

(1)
As a topological group, 
$G=\Aut(X_L/X)$ can be recovered from the topological quandle 
$Q=Q(X_L/X, \cM)$ 
as the closure $\overline{\Inn(Q)}$ of $\Inn(Q)$ in $C(Q, Q)$. 

(2)
The following three objects are the same: 
(a) $\overline{\Inn(Q)}$-orbits, 
(b) closures of  $\Inn(Q)$-orbits, and 
(c) inverse images of elements of $\cM$. 
\end{corollary}
\begin{proof}
(1) 
This is an immediate consequence of the proposition. 

(2)
Let $x$ be an element of $Q$. 
Since the action is continuous, 
we have $\overline{\Inn(Q)}\cdot x\subseteq \overline{\Inn(Q)\cdot x}$. 
The set 
$\overline{\Inn(Q)}\cdot x$ contains $\Inn(Q)\cdot x$ 
and is closed since $\overline{\Inn(Q)}\cong \Aut(X_L/X)$ is compact. 
Thus $\overline{\Inn(Q)}\cdot x\supseteq \overline{\Inn(Q)\cdot x}$ holds. 

Let us write $\pi$ for $\pi_{L, K}$. 
The inclusion 
$\Inn(Q)\cdot x\subseteq \pi^{-1}(\pi(x))$ is obvious, 
and since $\cM$ is given the discrete topology, 
$\pi^{-1}(\pi(x))$ is closed and 
$\overline{\Inn(Q)\cdot x}\subseteq \pi^{-1}(\pi(x))$ holds. 
Let $y$ be an element of $\pi^{-1}(\pi(x))$. 
Then, for any finite Galois intermediate field $M$, 
the point $\pi_{L, M}(y)$ is in the $G$-orbit of $\pi_{L, M}(x)$. 
Thus $y$ is in the closure of $G\cdot x$, which is equal to $\overline{\Inn(Q)\cdot x}$. 
\end{proof}
\begin{remark}
We see from (2)  that $\cM$ as a set and the map $Q\to \cM$ can be recovered from $Q$. 
Note that at this stage we do not know 
when and how the embedding of $\cM$ into $\Spec\cO_K$ can be recovered from $Q$. 
\end{remark}

\begin{corollary}\label{cor_recover_abelian}
Assume that $\cM$ has density $1$ 
and let $K^{\mathrm{ab}}$ denote the 
maximal abelian extension of $K$. 
Then the topological quandle 
$Q(X_{L\cap K^{\mathrm{ab}}}/X, \cM)$ 
can be described in terms of $Q$ as 
$Q/\overline{[\Inn(Q), \Inn(Q)]}$. 

In particular, 
$Q^{\mathrm{ab}}(X, \cM)$ can be described in terms of $Q(X, \cM)$. 
\end{corollary}
\begin{proof}
We have $\overline{[G, G]}=\overline{[\overline{\Inn(Q)}, \overline{\Inn(Q)}]}
= \overline{[\Inn(Q), \Inn(Q)]}$, and hence the assertion follows. 
\end{proof}

\subsection{The case $S=\{\fp\}$; $p$-adic logarithms and transcendency results}
Let us study the problem of reconstruction from $Q^{\mathrm{ab}}(X, \cM)$ 
in the special case where $S$ consists of one point $\fp$. 

As in the case of reconstruction of curves 
from the arithmetic fundamental groups, 
it seems reasonable to require that 
the fundamental group be sufficiently big. 

Let $Q$ and $G$ denote $Q^{\mathrm{ab}}(X, \cM)$ and 
$\pi_1^{\mathrm{ab}}(X)=\Aut(X^{\mathrm{ur}, \mathrm{ab}}/X)$, respectively. 

\begin{proposition}\label{prop_recover_p}
If $\Inn(Q)$ is infinite, 
then $\Char \bF(\fp)$ is characterized as 
the unique prime number $p$ such that 
the $p$-part of the profinite group $\overline{\Inn(Q)}$ is infinite. 

In particular, if $\cM$ has density $1$ and $G$ is infinite, 
then $p=\Char \bF(\fp)$ and the rank of the $p$-part of $G$ 
can be recovered from $Q$. 
\end{proposition}
\begin{proof}
By Class Field Theory, there is an exact sequence 
\[
1\to \cO_{K_\fp}^\times/\overline{\cO_K^{\times, +}} \to G\to \Clp(K)\to 1, 
\]
where $K_\fp$ is the completion of $K$ at $\fp$, 
$\cO_{K_\fp}$ its ring of integers, 
$\cO_K^{\times, +}$ is the group of totally positive units in $K$, 
$\overline{\cO_K^{\times, +}}$ is its closure in $\cO_{K_\fp}$ 
and $\Clp(K)$ is the narrow class group of $K$. 
Thus $\Char \bF(\fp)$ is the unique prime number $p$ 
such that the $p$-part of $G$ is infinite. 

By Proposition \ref{prop_continuity_action_hom}, 
$\overline{\Inn(Q)}$ is a closed subgroup of a quotient of $G$. 
So its $p$-part is finite if $p\not=\Char \bF(\fp)$, 
and infinite if $p=\Char \bF(\fp)$ and $\Inn(Q)$ is infinite. 

If $\cM$ has density $1$, then $G\cong\overline{\Inn(Q)}$ 
and the second assertion follows. 
\end{proof}

By the exact sequence above, 
the abelianized fundamental group $G$ can be 
linearized up to torsions 
using the $p$-adic logarithm function. 
In the proof of our main theorem, 
a key role is played 
by certain results on independence of $p$-adic logarithms 
of algebraic numbers. 
Let us collect the results from $p$-adic transcendental number theory 
that will be needed. 

Let $k$ be a non-archimedean local field of characteristic $0$. 
The power series $f(z):=\sum_{n=1}^\infty (-1)^{n-1}(z-1)^n/n$ 
converges on a neighborhood of $1$ in $k$. 
It can be extended to a function $\ln: \cO_k^\times\to k$ 
by defining $\ln(z)$ to be $f(z^N)/N$ for a sufficiently divisible integer $N$. 
It has a local inverse map near $0$ given by the power series 
$\exp(z):=\sum_{n=0}^\infty z^n/n!$.

\begin{theorem}[$p$-adic Six Exponentials Theorem, Th\'eor\`eme 1 of \cite{Serre1966}]
\label{thm_six_exponentials}
Let $k$ be a non-archimedean local field of characteristic $0$. 
For positive integers $d$ and $l$ satisfying 
$d>l/(l-1)$ (i.e. $dl>d+l$), 
let $x_1, \dots, x_d$ be elements of $k$ 
and $Y$ a free abelian subgroup of $k$ of rank $l$. 

If $\exp(x_iy)$ converges and is algebraic over $\bQ$ for any $i=1, \dots, d$ and $y\in Y$, 
then $x_1, \dots, x_d$ are linearly dependent over $\bQ$. 
\end{theorem}
This theorem is often used in the following form, hence the name. 
\begin{corollary}\label{cor_six_exponentials}
For $1\leq i\leq 2$ and $1\leq j\leq 3$, 
let $\alpha_{ij}\in\cO_k^\times$ be algebraic over $\bQ$. 
If the rows and columns of the matrix $M:=(\ln\alpha_{ij})$ 
are linearly independent over $\bQ$, 
then $M$ is of rank $2$ (as a matrix over $k$). 
\end{corollary}
\begin{proof}
Assume that $M$ is of rank $1$. 
Then one can find $x_i, y_j\in k$ such that $\ln \alpha_{ij}=x_iy_j$. 
The rank of $Y:=\langle y_1, y_2, y_3\rangle$ is $3$ 
by the assumption that the columns of $M$ are linearly independent 
over $\bQ$. 
The Six Exponentials Theorem implies that $x_1$ and $x_2$ are 
linearly dependent over $\bQ$, 
but this contradicts the assumption 
that the rows of $M$ are linearly independent over $\bQ$. 
\end{proof}

The Six Exponentials Theorem was generalized in the following form. 
\begin{theorem}[Th\'eor\`eme 2.1.p of \cite{Waldschmidt1981}]
\label{thm_independence}
Let $k$ be a non-archimedean local field of characteristic $0$. 
For positive integers $d$ and $l$, 
let $\alpha_{ij}\in\cO_k^\times$ be algebraic over $\bQ$ 
for $1\leq i\leq d$ and $1\leq j\leq l$. 

Then the rank $r$ of $M:=(\ln\alpha_{ij})$ satisfies 
$r\geq d\cdot\theta(M)/(1+\theta(M))$, 
where $\theta(M)$ is defined in \cite[\S7]{Waldschmidt1981}.

(See \cite[Th\'eor\`eme 1.1.p]{Waldschmidt1981} 
for a formulation similar to Theorem \ref{thm_six_exponentials}.) 
\end{theorem}

We do not write down the definition of $\theta(M)$ here, 
but we remark that 
it is equal to $l/d$ if the following holds: 
For any nonzero vectors 
$\boldsymbol{x}\in\bZ^d$ and $\boldsymbol{y}\in\bZ^l$, 
${}^t\!\boldsymbol{x}M\boldsymbol{y}\not=0$. 
In this case the theorem states 
that $\mathrm{rank} M\geq dl/(d+l)$.

\subsection{Rational and quadratic fields}

Now we furthermore assume that 
$K$ is either $\bQ$ or a quadratic field. 
We denote the characteristic of $\bF(\fp)$ by $p$ 
and the rank of the $p$-part of $G$ by $R$. 
For a prime ideal $\fl$ of $K$, 
write $e(\fl)$ and $f(\fl)$ for the absolute ramification index and 
the degree of extension of $\bF(\fl)$ over the prime field, 
respectively. 
There are $5$ cases. 

\begin{itemize}
\item[(1)]
If $K=\bQ$, then we have 
$G\cong \bZ_p^\times$, so $R=1$. 
\item[(2)]
Let $K$ be a quadratic field. 
\begin{itemize}
\item[(2-0)]
If $K$ is real and $e(\fp)=f(\fp)=1$, then $G$ is finite, 
so we will not consider this case. 
\item[(2-1)]
If $K$ is real and $(e(\fp), f(\fp))=(1, 2)$ or $(2, 1)$, 
then the unit group has rank $1$ and $R=1$. 
\item[(2-2)]
If $K$ is complex and $e(\fp)=f(\fp)=1$, 
we again have $R=1$. 
\item[(2-3)]
If $K$ is complex and $(e(\fp), f(\fp))=(1, 2)$ or $(2, 1)$, 
we have $R=2$. 
\end{itemize}
\end{itemize}

\begin{example}
Let us give an explicit description of the quandle 
$Q=Q^{\mathrm{ab}}(\Spec\bZ\setminus\{(p)\}, \cM)$  
using Proposition \ref{prop_galois_quandle_representation}. 
We identify the maximal ideals of $\bZ$ with their positive generators, 
i.e. prime numbers. 
The fiber of $\pi: Q\to \cM$ over $l\in \cM$ can be 
identified with $\bZ_p^\times/\overline{\langle l \rangle}$ 
under a choice of a point in $\pi^{-1}(l)$, 
so we have $Q=\coprod_{l\in \cM} \bZ_p^\times/\overline{\langle l \rangle}$. 
Note that each fiber of $\pi$ is a finite set. 
(The same is true in the cases (2-0), (2-1) and (2-2). 
In the case (2-3), each fiber of $\pi$ is an infinite set.)

The quandle operation is given by 
$x\rhd y=\pi(x)y$, 
where $\pi(x)$ is regarded as a rational prime 
and acts on each $\bZ_p^\times/\overline{\langle l \rangle}$ 
by multiplication. 

We can visualize the quandle $Q$ as a kind of slot machine: 
A pair of a reel and a button is associated to each $l\in \cM$. 
When a button is pressed, 
reels rotate by certain angles. 
(Strictly speaking, 
$\bZ_p^\times/\overline{\langle l \rangle}$ can be non-cyclic when $p=2$ 
and the word ``rotation'' might be inappropriate.) 
\end{example}

Now we state our main theorem.

\begin{theorem}\label{thm_recover}
Let each of $K$ and $K'$ be either $\bQ$ or a quadratic field, 
$\fp$ and $\fp'$ primes in $K$ and $K'$ respectively, 
and $\cM$ and $\cM'$ sets of primes with density $1$ in $K$ and $K'$, 
respectively. 

Write $X=\Spec\cO_K\setminus\{\fp\}$ 
and $Q=Q^{\mathrm{ab}}(X, \cM)$, 
with the natural maps $\varpi: Q\to X$ 
and $\rho: Q\to\Spec\bZ$. 
Define $X', Q', \varpi'$ and $\rho'$ in the same way. 
Assume that $\pi_1^{\mathrm{ab}}(X)$ or $\pi_1^{\mathrm{ab}}(X')$ is infinite. 

If $\varphi: Q\to Q'$ is an isomorphism of topological quandles, 
then there exists an isomorphism $\sigma: K\to K'$ 
which maps $\fp$ to $\fp'$, 
and $\varphi$ commutes with $\rho$ and $\rho'$. 

Furthermore, 
except in the case where $K$ is real quadratic, 
$\sigma$ can be chosen 
so that $\varphi$ is compatible with the isomorphism 
$X\to X'$ induced by $\sigma$. 
\end{theorem}
\begin{remark}
A similar statement holds for $Q(X, \cM)$ and $Q(X', \cM')$
by Corollary \ref{cor_recover_abelian}. 
\end{remark}

Let us begin the proof of the theorem. 
Write $p$ for the characteristic of $\bF(\fp)$, 
which can be recovered from $Q$ by Proposition \ref{prop_recover_p}
and is equal to the characteristic of $\bF(\fp')$. 
Let $G=\pi_1^{\mathrm{ab}}(X)$ 
and write $G_p$ for its $p$-part. 
The latter is a $\bZ_p$-module in a natural way, 
so let $V=G_p\otimes_{\bZ_p}\bQ_p$. 
Write $G'$, $G'_p$ and $V'$ for the corresponding groups 
associated to $X'$. 

The groups $G$ and $G'$, 
the sets $\cM$ and $\cM'$ 
and the maps $\pi: Q\to\cM$ and $\pi': Q'\to\cM'$ 
can be recovered from $Q$ and $Q'$ 
by Corollary \ref{cor_recover_group_and_primes}. 
(We are making distinction between $\pi$ and $\varpi$ 
since $X$ is not yet reconstructed and 
sometimes we cannot recover the inclusion $\cM\to X$.) 

We also have maps $r: \cM\to V$ and $r': \cM'\to V'$ 
given by the quandle operation: 
For $a\in\cM$, choose $x\in\pi^{-1}(a)$ and 
let $r(a) = (\hbox{$p$-part of $s_x$})\otimes 1$. 
These are the maps induced by the reciprocity maps 
in Class Field Theory. 

Thus the theorem is reduced to the following. 
\begin{theorem}
Let the notations be as above 
and assume that $V$ is nonzero. 

Let $\bar{\varphi}: \cM\to\cM'$ be a bijection 
and $\psi: V\to V'$ an isomorphism of $\bQ_p$-vector spaces 
satisfying $r'\circ \bar{\varphi}=\psi\circ r$. 

Then there exists an isomorphism $\sigma: K\to K'$ 
which maps $\fp$ to $\fp'$, 
and the residue characteristics $c(\fl)$ and $c(\bar{\varphi}(\fl))$ coincide 
for any $\fl\in\cM$. 

Furthermore, 
except in the real quadratic case, 
$\sigma$ can be chosen so that $\sigma(\fl)=\bar{\varphi}(\fl)$. 
\end{theorem}

Note that, 
if we start with the description of $G$ by Class Field Theory, 
there is a natural coordinate on $V$ given by the $p$-adic logarithm function. 
Here we start with $Q$, 
so there is no natural coordinate on $V$ 
and we cannot recover $\fl$ from $ r(\fl)$ directly. 
We can however take a number of prime ideals 
and compare the configurations of their images in $V$ and $V'$. 

If $K$ is a number field and $\fp$ is a prime in $K$, 
we denote the localization of $\cO_K$ at $\fp$ by $\cO_{K, \fp}$ 
and the composite map 
$\cO_{K, \fp}^\times\to\cO_{K_\fp}^\times\overset{\ln}{\to} K_\fp$ 
by $\Lnpe$. 
(We have to be a little careful about where we take the logarithms.) 

The following fact is easy to see. 
\begin{lemma}\label{lem_independence}
Let $k$ be a non-archimedean local field of characteristic $0$, 
and $\alpha_1, \dots, \alpha_n$ elements of $\cO_k^\times$. 
Then $\ln \alpha_1, \dots, \ln \alpha_n$ 
are linearly independent over $\bQ$ 
if and only if $\alpha_1, \dots, \alpha_n$ are 
multiplicatively independent. 

In particular, 
if $K$ is a number field, 
$\fp$ is a prime in $K$, 
and $\alpha_1, \dots, \alpha_n$ are elements of $\cO_{K, \fp}^\times$ 
such that 
the ideals $\alpha_1\cO_K, \dots, \alpha_n\cO_K$ are 
multiplicatively independent, 
then $\Lnpe\alpha_1, \dots, \Lnpe\alpha_n$ 
are linearly independent over $\bQ$. 
\end{lemma}

Denote by $R$ the dimension of $V$, 
which is equal to the dimension of $V'$. 

\subsection{Case $R=1$} 
Let $\fl_1$ and $\fl_2$ be elements of $\cM$. 
If $V$ is $1$-dimensional, 
the quantity $ r(\fl_2)/ r(\fl_1)\in\bQ_p$ 
is a well-defined invariant 
in the sense that if $\bar{\varphi}(\fl_1)=\fl'_1$ and 
$\bar{\varphi}(\fl_2)=\fl'_2$ 
then $ r(\fl_2)/ r(\fl_1)= r'(\fl'_2)/ r'(\fl'_1)$ holds. 
Let us express $r(\fl_2)/r(\fl_1)$ 
in terms of $p$-adic logarithms. 

\medskip
(1)
If $K=\bQ$, 
then the image of $\fl\in\cM$ in $G$ can be identified with $N(\fl)=c(\fl)$ 
under the identification of $G$ with $\bZ_p^\times$. 
Thus 
\[
 r(\fl_2)/ r(\fl_1) = \Lnp N(\fl_2)/\Lnp N(\fl_1). 
\]

(2-1)
$K$ is real and $e(\fp)f(\fp)=2$. 
Let $\gamma$ be a unit which is not a root of unity. 
Then $V=((\cO_{K_\fp}^\times)_p\otimes_{\bZ_p}\bQ_p)/\bQ_p\gamma_p$, 
where $\gamma_p$ is the $p$-component of $\gamma$ 
in $\cO_{K_\fp}^\times$. 
The $p$-adic logarithm function induces 
$(\cO_{K_\fp}^\times)_p\otimes_{\bZ_p}\bQ_p\cong K_\fp$, 
and satisfies $\Lnpe \bar{z}=\overline{\Lnpe z}$ 
since $\bar{\fp}=\fp$. 
Here $\bar{z}$ is the conjugate of $z\in K_\fp$ over $\bQ_p$. 
If $K=\bQ(\sqrt{m})$, 
an isomorphism between $K_\fp$ and $\bQ_p^2$ 
as $\bQ_p$-vector spaces is given by 
$z\mapsto (z+\bar{z}, (z-\bar{z})/\sqrt{m})$. 
Composing this map with $\Lnpe$, we obtain a map 
$\alpha\mapsto (\Lnpe N(\alpha), (\Lnpe \alpha\bar\alpha^{-1})/\sqrt{m})$. 
This maps $\gamma$ to $(0, 2\ \Lnpe \gamma/\sqrt{m})$. 
Thus $\Lnpe N(\alpha)=\Lnp N(\alpha)$ 
gives a linear coordinate of the image of $\alpha$ in $V$. 

Now let $\fl_1$ and $\fl_2$ be elements of $\cM$. 
One can find a positive integer $n$ such that $\fl_1^n$ and $\fl_2^n$ are 
principal ideals generated by totally positive numbers 
$\alpha_1$ and $\alpha_2$. 
Then we have 
\[
 r(\fl_2)/ r(\fl_1) 
= \Lnp N(\alpha_2)/\Lnp N(\alpha_1)
= \Lnp N(\fl_2)/\Lnp N(\fl_1). 
\]

(2-2)
$K$ is complex and $e(\fp)f(\fp)=1$. 
Then $V=(\cO_{K_\fp}^\times)_p\otimes_{\bZ_p}\bQ_p$. 
Let $\fl_1$ and $\fl_2$ be elements of $\cM$. 
One can find a positive integer $n$ such that $\fl_1^n$ and $\fl_2^n$ 
are generated by $\alpha_1$ and $\alpha_2$, 
and then we have 
\[
 r(\fl_2)/ r(\fl_1) = \Lnpe \alpha_2/\Lnpe \alpha_1. 
\]

\medbreak
In any case, 
one can take any number of prime ideals $\fl_i$ ($i=1, 2, \dots$) 
from $\cM$ 
with different residue characteristics 
and with $e(\fl_i)=f(\fl_i)=1$ 
since $\cM$ has density $1$. 
Let $\fl'_i$ be the corresponding elements of $\cM'$.

First consider the case where $K$ is of type (2-2). 
Take a positive number $n$ such that $\fl_i^n$ are principal 
and let $\alpha_i$ denote generators. 

If $K'$ is also of type (2-2), 
let $\alpha'_1$, $\alpha'_2$ and $\alpha'_3$ be defined in the same way. 
If $K\cong K'$, identify them by the isomorphism $\sigma$ 
which maps $\fp$ to $\fp'$ and write $\fP$ for $\fp$. 
If $K\not\cong K'$, 
then $KK'$ would be of degree $4$ over $\bQ$. 
Let $\fP$ be the prime ideal in $KK'$ over $\fp$ and $\fp'$. 
In either case, the logarithms at $\fp$ and $\fp'$ can be 
regarded as the logarithms at $\fP$. 
If $K'$ were of type (1) or (2-1), 
write $\alpha'_i = N(\fl'_i)$.
In this case logarithms at $(p)$ is equal to the logarithm at $\fp$. 
Write $\fP$ for $\fp$. 

In any case, we have 
\[
\mathrm{rank}
\begin{pmatrix}
\Ln^\fP \alpha_1 & \Ln^\fP \alpha_2 & \Ln^\fP \alpha_3 \\
\Ln^\fP \alpha'_1 & \Ln^\fP \alpha'_2 & \Ln^\fP \alpha'_3 
\end{pmatrix}
=1. 
\]
By Corollary \ref{cor_six_exponentials} 
the rows or the columns must be linearly dependent over $\bQ$. 
Since $\fl_i$ are multiplicatively independent
the columns are independent by Lemma \ref{lem_independence}. 
Thus $\Ln^\fP \alpha_1$ and $\Ln^\fP \alpha'_1$ 
must be dependent over $\bQ$, 
and $\alpha_1$ and $\alpha'_1$ 
are multiplicatively dependent. 
It follows that $\alpha_1^m$ is contained in $K'$ for some $m\not=0$. 
We see from $e(\fl_1)=f(\fl_1)=1$ that $\alpha_1^m$ generates $K$ over $\bQ$,  
so $K'\supseteq K$, hence $K=K'$, holds. 

The dependence of rows over $\bQ$ implies 
that $\fl_i$ and $\fl'_i$ are multiplicatively dependent, 
hence that $\fl_i=\fl'_i$. 
Since $V$ is a $1$-dimensional vector space, 
$\psi$ is the identity map 
(under the identification by $\sigma$). 
Since $\cM\to (\cO_{K_\fp}^\times)_p\otimes_{\bZ_p}\bQ_p$ 
is always injective and 
$V$ can be identified with $(\cO_{K_\fp}^\times)_p\otimes_{\bZ_p}\bQ_p$ 
in this case, 
we see that $\bar{\varphi}$ is the identity map. 

\medbreak
Next assume that $K$ is of type (1) or (2-1). 
Then $K'$ is also of type (1) or (2-1). 
We show that the map $c$ and the field $K$ can be 
recovered from $ r: \cM\to V$. 
The residue characteristic $c(\fl_1)$ 
is the unique rational prime number $l_1$ such that 
\begin{eqnarray*}
 r(\fl_2)/r(\fl_1) & \equiv & \Lnp l_2/\Lnp l_1 \mod \bQ^\times \quad\textrm{ and } \\
 r(\fl_3)/r(\fl_1) & \equiv & \Lnp l_3/\Lnp l_1 \mod \bQ^\times 
\end{eqnarray*}
hold for certain primes $l_2$ and $l_3$. 
In fact, these relations imply  
\[
\mathrm{rank}
\begin{pmatrix}
\Lnp N(\fl_1) & \Lnp N(\fl_2) & \Lnp N(\fl_3) \\
\Lnp l_1 & a\cdot\Lnp l_2 & b\cdot\Lnp l_3 
\end{pmatrix}
=1 
\]
for some rational numbers $a$ and $b$. 
Since we chose $\fl_i$ to have distinct residue characteristics, 
the columns are independent over $\bQ$. 
By Corollary \ref{cor_six_exponentials}, 
the rows are dependent, and 
$N(\fl_1)$ is a power of $l_1$. 

The set $\{l\in\Spec\bZ \mid \#c^{-1}(l)=2\}$ is of density $1/2$ (resp. empty) 
if $K$ is quadratic (resp. $K=\bQ$), 
and this set determines $K$.

\subsection{Case $R=2$} 
This is the case where $K$ is complex and $e(\fp)f(\fp)=2$. 

We provide two proofs. 
The first one makes use of Theorem \ref{thm_independence} 
and is somewhat more direct. 
The second one is based on a characterization 
of a special subspace $\bQ_p\subset V\cong K_\fp$ 
and only uses Corollary \ref{cor_six_exponentials}. 

\begin{proof}[First proof]
In $KK'$, let $\fP$ be a prime over $\fp$ and $\fp'$. 

Let $\cM_1=\{\fl\in\cM\mid e(\fl)f(\fl)=1, \bar{\fl}\in\cM\}$ 
and $\fl$ an element of $\cM_1$. 
Let $n$ be a positive integer such that $\fl^n$ is principal, 
generated by $\alpha$. 
Then $\{\Lnpe \alpha, \Lnpe \bar{\alpha}\}$ spans $K_\fp$ over $\bQ_p$. 
In fact, if not, we would have $\Lnpe \alpha=z\cdot\Lnpe \bar{\alpha}$
for some $z\in\bQ_p$. 
Conjugating, we also have $\Lnpe \bar{\alpha}=z\cdot\Lnpe \alpha$ 
and so $z$ must be $\pm 1$. 
This contradicts Lemma \ref{lem_independence}. 

Let $\fl'=\bar{\varphi}(\fl)$ 
and $\fl''=\bar{\varphi}(\bar{\fl})$. 
Assume that both $\{\fl, \bar{\fl}, \fl'\}$
and $\{\fl, \bar{\fl}, \fl''\}$ 
are multiplicatively dependent 
in the ideal group of $KK'$. 
Then some powers of $\fl'$ and $\fl''$ are contained in 
$\langle \fl, \bar{\fl}\rangle$, 
and so $\fl'$ and $\fl''$ lie over $c(\fl)$. 
If $K\not\cong K'$, there would be $4$ primes in $KK'$ over $c(\fl)$, 
namely $\fl_1:=\fl+\fl'$, $\fl_2:=\fl+\fl''$, $\fl_3:=\bar{\fl}+\fl'$ and $\fl_4:=\bar{\fl}+\fl''$, 
and $\fl'=\fl_1\fl_3$ is not in the $\bQ$-span of 
$\fl=\fl_1\fl_2$ and $\bar{\fl}=\fl_3\fl_4$. 
Thus $K=K'$ 
and $\{\fl', \fl''\}=\{\fl, \bar{\fl}\}$. 
Since $\{r(\fl), r(\bar{\fl})\}$ spans $V\cong K_\fp$ over $\bQ_p$, 
$\psi$ is either the identity or the conjugation map, 
and hence similarly for $\bar{\varphi}$. 

Thus we suppose that 
either $\{\fl, \bar{\fl},\bar{\varphi}(\fl)\}$
or $\{\fl, \bar{\fl},\bar{\varphi}(\bar{\fl})\}$ 
is multiplicatively independent for any $\fl\in\cM_1$, 
and draw a contradiction. 

We will apply Theorem \ref{thm_independence} 
with $d=3$ and $l=7$. 
Choose $\fl_1, \dots, \fl_7\in\cM_1$ in the following way. 
Let $\fl_1$ be an element of $\cM_1$ and 
write $\fl'_1=\bar{\varphi}(\fl_1)$. 
From the above, we may assume that $\fl_1, \overline{\fl}_1$ and $\fl'_1$ 
are multiplicatively independent. 
When $\fl_1, \dots, \fl_{k-1}$ are chosen, 
let $\fl_k$ be an element of $\cM_1$ 
such that, denoting $\bar{\varphi}(\fl_k)$ by $\fl'_k$, 
the residue characteristics $c(\fl_k)$ and $c(\fl'_k)$ are both different from 
any of $c(\fl_i)$ and $c(\fl'_i)$, $1\leq i < k$, 
and that $\fl_k, \overline{\fl}_k$ and $\fl'_k$ 
are multiplicatively independent. 
We can carry this out since $\cM$ is assumed to be of density $1$. 

Let $n$ be a positive integer such that 
$\fl_i^n=(\alpha_i)$ and $(\fl'_i)^n=(\beta_i)$ holds. 
Then we have 
\[
\mathrm{rank}
\begin{pmatrix}
\Ln^\fP \alpha_1 & \Ln^\fP \alpha_2 & \dots & \Ln^\fP \alpha_7 \\
\Ln^\fP \bar{\alpha}_1 & \Ln^\fP \bar{\alpha}_2 & \dots & \Ln^\fP \bar{\alpha}_7 \\
\Ln^\fP \beta_1 & \Ln^\fP \beta_2 & \dots & \Ln^\fP \beta_7
\end{pmatrix}
\leq 2. 
\]
In fact, since $K_\fp$ is of dimension $2$ over $\bQ_p$, 
the entries of the first row can be written 
as a $\bQ_p$-linear combination of two of them, 
say $\Ln^\fP \alpha_i$ and $\Ln^\fP \alpha_j$. 
The other rows are obtained by applying $\bQ_p$-linear maps 
to the entries of the first row. 
Thus columns of the matrix can be written as a $\bQ_p$-linear combination 
of the $i$-th and $j$-th columns. 

The entries are linearly independent over $\bQ$: 
The prime ideals (in $K$ or $K'$) corresponding to entries in different columns 
belong to different residue characteristics, so a linear relation of entries 
implies a linear relation of entries in a column. 
But we assumed that entries in a column are linearly independent. 

Thus the condition stated after Theorem \ref{thm_independence}
is satisfied, 
and the theorem asserts that 
the rank is at least $3\cdot 7/(3+7)>2$, a contradiction. 
\end{proof}

\begin{proof}[Second proof]
We characterize the subspace $\bQ_p$ of $V\cong K_\fp$ 
in terms of the map $r: \cM\to V$. 
\begin{lemma}
Let us identify $V$ with $K_\fp$. 

(1)
The following conditions are equivalent 
for a $1$-dimensional $\bQ_p$-subspace $W$ of $V$. 
\begin{itemize}
\item[(a)]
$W=\bQ_p$. 
\item[(b)]
For any finite subset $E$ of $\cM$, 
the set 
$P:=\{\{\fl, \fl'\}\subset \cM\setminus E \mid  r(\fl)+ r(\fl')\in W, \fl\not=\fl'\}$ 
is nonempty (hence infinite). 
\item[(c)]
The set $P:=\{\{\fl, \fl'\}\subset \cM \mid  r(\fl)+ r(\fl')\in W, \fl\not=\fl'\}$ 
has at least $3$ elements. 
\end{itemize}

(2)
If $\fl\in \cM$ is such that $ r(\fl)\not\in\bQ_p$, 
there is at most one $\fl'\in \cM$ 
such that $ r(\fl)+ r(\fl')\in\bQ_p$, 
namely $\fl'=\bar{\fl}$ (assuming it is contained in $\cM$). 
\end{lemma}
\begin{proof}
(1)
(a) $\Rightarrow$ (b) 
We have $r(\bar{\fl})=\overline{r(\fl)}$, 
so $r(\fl)+ r(\bar{\fl})\in \bQ_p$ holds for any $\fl\in\cM$. 
Since $\cM$ has density $1$, 
one can find $\fl\in\cM\setminus (E\cup \overline{E})$ 
such that $\fl\not=\bar{\fl}$. 

(b) $\Rightarrow$ (c) is obvious. 

(c) $\Rightarrow$ (a) 
Let $\{\fl_1, \fl_2\}$, $\{\fl_3, \fl_4\}$ and $\{\fl_5, \fl_6\}$ 
be distinct sets of two primes 
such that 
$r(\fl_1)+ r(\fl_2)$, $r(\fl_3)+ r(\fl_4)$ and $r(\fl_5)+ r(\fl_6)$ 
are contained in $W$. 
Write $z_i$ for $r(\fl_i)=(\Lnpe \alpha_i)/n$, 
where $\fl_i^n$ is principal with generator $\alpha_i$. 
Note that $z,\ w\in K_\fp$ are linearly dependent over $\bQ_p$ 
if and only if 
\[
\mathrm{rank}
\begin{pmatrix}
z & w \\ 
\bar{z} & \bar{w} 
\end{pmatrix}
\leq 1. 
\]
Thus we have 
\[
\mathrm{rank}
\begin{pmatrix}
z_1+z_2 &  z_3+z_4 & z_5+z_6 \\
\bar{z}_1+\bar{z}_2 &  \bar{z}_3+\bar{z}_4 & \bar{z}_5+\bar{z}_6 
\end{pmatrix}
=1. 
\]
Assume that $a(z_1+z_2)+b(z_3+z_4)+c(z_5+z_6)=0$ holds 
for $a, b, c\in\bZ$. 
Then $(\fl_1\fl_2)^a(\fl_3\fl_4)^b(\fl_5\fl_6)^c=(1)$ holds. 
Write $k$ for $\#\{\fl_1, \dots, \fl_6\}$, 
and we have a relation of the form 
$a\boldsymbol{v}_1+b\boldsymbol{v}_2+c\boldsymbol{v}_3=
\boldsymbol{o}$, 
where $\boldsymbol{v}_i\in\bZ^k$ are pairwise distinct vectors 
with $2$ entries equal to $1$ and the other entries $0$. 
Then it is easy to see that $a=b=c=0$, 
and the columns of the above matrix are 
linearly independent over $\bQ$. 

Thus by Corollary \ref{cor_six_exponentials} 
the rows are linearly dependent over $\bQ$. 
It follows that $\fl_1\fl_2=\bar{\fl}_1\bar{\fl}_2$, 
and $z_1+z_2\in\bQ_p$ holds. 

(2)
If $r(\fl)+r(\fl')\in \bQ_p$, 
then $r(\fl)+r(\fl')=r(\bar{\fl})+r(\overline{\fl'})$ 
holds. 
It follows that $\fl\,\fl'=\bar{\fl}\,\overline{\fl'}$. 
Assuming that $r(\fl)\not\in\bQ_p$, 
$\fl\not=\bar{\fl}$ and hence 
$\fl'=\bar{\fl}$ holds. 
\end{proof}

\begin{remark}
In (1), 
we also have the following sufficient condition: 
The set $\{\fl\in\cM \mid  r(\fl)\in W\}$ is infinite (or has at least $3$ elements). 

If we assume that $\cM$ contains infinitely many 
(or at least $3$) primes $\fl$ 
with $e(\fl)f(\fl)=2$ (or equivalently $\fl=\bar{\fl}$), 
it is also a necessary condition. 
\end{remark}

By this lemma 
the subspace $W$ corresponding to $\bQ_p$ can be recovered 
from $r: \cM\to V$. 
The set $\cM\cap\overline{\cM}$ can be characterized 
as the set of $\fl\in\cM$ satisfying the condition that 
either $r(\fl)\in W$ holds 
or there exists $\fl'\in \cM$ such that $r(\fl)+r(\fl')\in W$. 
For $\fl\in\cM\cap\overline{\cM}$, 
its conjugate $\bar{\fl}$ is characterized as follows: 
If $r(\fl)\in W$, then $\overline{\fl}=\fl$. 
If not, $\bar{\fl}$ is the unique $\fl'\in\cM$ 
such that $r(\fl)+r(\fl')\in W$.

Thus for $\fl_1, \fl_2\in\cM\cap\overline{\cM}$, 
\[
(r(\fl_2)+ r(\bar{\fl}_2))/(r(\fl_1)+ r(\bar{\fl}_1))
= \Lnp N(\fl_2)/\Lnp N(\fl_1)
\]
was recovered from $r: \cM\to V$. 

As in the case of type (1) or (2-1), 
we can recover the residue characteristics 
$c: \cM\cap\overline{\cM}\to \Spec\bZ$ 
and hence a set of rational primes of density $1/2$ which split in $K$. 
Thus $K$ is determined, 
and one can find an isomorphism of $K$ and $K'$. 
The isomorphism $\psi$ induces the identity map on $W$. 
If $\fl\in\cM\cap\overline{\cM}$ is such that $ r(\fl)\not\in W$, 
then $c(\bar{\varphi}(\fl))=c(\fl)$, 
so $\bar{\varphi}(\fl)$ is either $\fl$ or $\overline{\fl}$. 
Thus, choosing the other isomorphism of $K$ and $K'$ if necessary, 
we may assume that $\bar{\varphi}(\fl)=\fl$. 
Now the linear map $\psi: V\to V'$ is the identity map, 
and so is $\cM\to \cM'$ since $ r$ is injective. 
\end{proof}

\section{Automorphism groups}

For a topological quandle $Q$, 
let us denote by $\Aut(Q)$ 
the group of homeomorphic automorphisms of $Q$. 
Our aim in this section is to study $\Aut(Q)$ 
for topological quandles $Q$ associated to Galois covers. 

Let $X$ be a normal, separated and integral scheme 
of finite type over $\bZ$, 
$\cM$ a set of closed points of $X$ 
and $\tilde{X}$ a Galois cover of $X$ unramified over $\cM$, 
possibly of infinite degree. 
Denote $Q(\tilde{X}/X, \cM)$ by $Q$ 
and $\Aut(\tilde{X}/X)$ by $G$. 
Write $\pi: Q\to\cM$ and $c: \cM\to\Spec\bZ$ for the natural maps. 

If $G$ is abelian 
and $a$ is an element of $\cM$, 
the Frobenius automorphism associated to $x\in \pi^{-1}(a)$ is 
independent of $x$, 
so let $s_a$ denote it. 

\begin{proposition}
(1) 
If $G$ is abelian, 
a choice of a section of $\pi: Q\to\cM$ 
determines an action of $\Aut(X, \cM)$ 
on $Q$. 

(2)
If $G$ is abelian, 
then there is a natural faithful action of 
$\prod_{a\in\cM} G/\overline{\langle s_a\rangle}$ 
on $Q$. 

(3)
If $X=\Spec\cO_K\setminus S$ for a number field $K$ 
and a finite set $S$ of primes 
and $\cM$ has density $1$, 
then there is a natural homomorphism 
$\Aut(Q)\to S_\cM$, 
where $S_\cM$ is the (discrete) group of 
bijections from $\cM$ to $\cM$. 

If $G$ is abelian, 
then the kernel is isomorphic to $\prod_{a\in\cM} G/\overline{\langle s_a\rangle}$ 
via the action in (2). 
In other words, 
there is a natural exact sequence 
\[
1\to \prod_{a\in\cM} G/\overline{\langle s_a\rangle}\to\Aut(Q)
\to S_\cM. 
\]
\end{proposition}

\begin{proof}
Recall 
that $Q$ is isomorphic to 
$Q(G, \{\overline{\langle s_{x_a}\rangle}\}_{a\in\cM}, \{s_{x_a}\}_{a\in\cM})$ 
where $a\mapsto x_a$ is a section of $\pi$ 
(Proposition \ref{prop_galois_quandle_representation}). 

(1)
If $G$ is abelian, 
the group $\Aut(X, \cM)$ acts on $G$ and $\cM$ 
with the property $\sigma(s_a)=s_{\sigma(a)}$ 
for $\sigma\in\Aut(X, \cM)$ and $a\in \cM$. 
Thus the assertion follows from Lemma \ref{lem_auto} (1). 

(2) 
This is immediate from Lemma \ref{lem_auto} (2). 

(3)
Apply Lemma \ref{lem_auto} (3) 
noting that $\Inn(Q)$ is dense in $\gamma(G)$ 
by Corollary \ref{cor_recover_group_and_primes}. 
\end{proof}

\begin{lemma}\label{lem_auto}
Let $G$ be a topological group, 
$\{z_\lambda\}_{\lambda\in\Lambda}$ a family of elements of $G$ 
and $\{H_\lambda\}_{\lambda\in\Lambda}$ a family of closed 
subgroups of $G$ 
such that $z_\lambda\in H_\lambda$ and 
$H_\lambda$ is contained in the centralizer of $z_\lambda$. 
Let $Q=Q(G, \{H_\lambda\}, \{z_\lambda\})$.

(1)
If $G'$ is a group acting on $G$ and $\Lambda$ 
with $\sigma(z_\lambda)=z_{\sigma(\lambda)}$ 
and $\sigma(H_\lambda)=H_{\sigma(\lambda)}$ 
for any $\sigma\in G'$ and $\lambda\in\Lambda$, 
then $G'$ acts on $Q$ by $\sigma(xH_\lambda):=\sigma(x)H_{\sigma(\lambda)}$. 

(2)
If $G$ is abelian, 
then there is a faithful action of 
$\prod_{\lambda\in\Lambda}G/H_\lambda$ on $Q$. 

If $(Q', G)$ is an augmented quandle, 
$\Lambda=G\backslash Q'$, 
$\{q_\lambda\}_{\lambda\in\Lambda}$ is a complete system of representatives, 
$H_\lambda=\Stab_G(q_\lambda)$ and $z_\lambda=\varepsilon(q_\lambda)$, 
then the action of $\prod_{\lambda\in\Lambda}G/H_\lambda$ on $Q'$ 
induced via the isomorphism of  Proposition \ref{prop_quandle_representation} 
does not depend on the choice of $q_\lambda$. 

(3)
Assume that $G$ is compact 
and write $\gamma: G\to C(Q, Q)$ 
for the map defined by the left translations. 
If $\Inn(Q)$ is dense in $\gamma(G)$, 
then there is a natural homomorphism 
$\Aut(Q)\to S_\Lambda$. 

If $G$ is abelian, 
then the kernel is isomorphic to $\prod_{\lambda\in\Lambda}G/H_\lambda$ 
via the action in (2). 
\end{lemma}
\begin{proof}
Recall that $Q=\coprod_{\lambda\in\Lambda}G/H_\lambda$ 
(Example \ref{ex_quandle_representation}). 

(1)
Note that $\sigma(x)H_{\sigma(\lambda)}$ is equal to 
the set-theoretic image $\sigma(xH_\lambda)$ from the assumption 
that $\sigma(H_\lambda)=H_{\sigma(\lambda)}$. 
Thus the action is well-defined. 

We have 
\[
\sigma(xH_\lambda\rhd yH_\mu)=\sigma(xz_\lambda x^{-1}yH_\mu)
=\sigma(xz_\lambda x^{-1}y)H_{\sigma(\mu)}
\]
and 
\[
\sigma(xH_\lambda)\rhd\sigma(yH_\mu)=\sigma(x)H_{\sigma(\lambda)}\rhd\sigma(y)H_{\sigma(\mu)}
= \sigma(x)z_{\sigma(\lambda)}\sigma(x)^{-1}\sigma(y)H_{\sigma(\mu)}. 
\]
These are equal from the assumption $\sigma(z_\lambda)=z_{\sigma(\lambda)}$, 
so $G'$ acts on $Q$ by quandle automorphisms. 

(2)
For $u=(u_\lambda H_\lambda)\in\prod_{\lambda\in\Lambda}G/H_\lambda$, 
we define $u\cdot xH_\lambda:=u_\lambda xH_\lambda$. 
Noting that $xH_\lambda\rhd yH_\mu=z_\lambda yH_\mu$, 
it is easy to see that $\prod_{\lambda\in\Lambda}G/H_\lambda$ acts on $Q$ 
by quandle automorphisms. 
This action is obviously faithful. 

When $Q$ comes from an augmented quandle $Q'$, 
first note that $\varepsilon(q_\lambda)$ and $\Stab_G(q_\lambda)$ 
are independent of the choice of $q_\lambda$. 
The isomorphism 
$\varphi_q: Q\overset{\sim}{\to} Q'$ 
depends on the choice of $\{q_\lambda\}$, 
and if $\{q'_\lambda\}$ is another system of representatives 
related with $\{q_\lambda\}$ by $q'_\lambda=g_\lambda q_\lambda$, $g_\lambda\in G$, 
then $\varphi_q^{-1}\circ\varphi_{q'}$ is equal to the translation by $(g_\lambda H_\lambda)$. 
This commutes with the action of  $\prod_{\lambda\in\Lambda}G/H_\lambda$, 
so the actions on $Q'$ induced by $\varphi_q$ and $\varphi_{q'}$ 
are the same.

(3)
Note that $s_{xH_\lambda}$ is the left translation by $xz_\lambda x^{-1}$, 
so we always have $\Inn(Q)\subseteq\gamma(G)$. 

Let $\varphi$ be an element of $\Aut(Q)$ 
and $q=xH_\lambda$ an element of $Q$. 
Then we have $\varphi(\Inn(Q)q)\subseteq \Inn(Q)\varphi(q)$. 
If $\Inn(Q)$ is dense in $\gamma(G)$, 
then we have $\varphi(G/H_\lambda)=\varphi(Gq)\subseteq G \varphi(q)$ by continuity. 
We see the other inclusion by considering $\varphi^{-1}$. 
Thus $\varphi$ maps $G/H_\lambda$ to some $G/H_\mu$, 
and induces a permutation on $\Lambda$. 

The kernel of $\Aut(Q)\to S_\Lambda$ 
consists of automorphisms preserving each $G/H_\lambda$. 
Assume that $G$ is abelian and 
let $\varphi$ be such an automorphism. 
If $q\in G/H_\lambda\subseteq Q$, 
then $s_q$ is the left translation by $z_\lambda$. 
By assumption, $\varphi(q)$ is in the same $G/H_\lambda$, 
so we have $s_q=s_{\varphi(q)}$. 
From $\varphi\circ s_q=s_{\varphi(q)}\circ\varphi$ 
it follows that $\varphi$ commutes with $\Inn(Q)$. 
By continuity it also commutes with $G$. 
Let $u_\lambda$ be such that 
$\varphi(H_\lambda)=u_\lambda H_\lambda\in G/H_\lambda$ 
and write $u=(u_\lambda H_\lambda)\in\prod_{\lambda\in\Lambda}G/H_\lambda$. 
Then we have 
$\varphi(xH_\lambda)=x\varphi(H_\lambda)=xu_\lambda H_\lambda=
u_\lambda xH_\lambda=u\cdot xH_\lambda$ 
for any $x\in G$ and $\lambda\in\Lambda$, 
so $\varphi$ coincides with the action of $u$. 
\end{proof}

\begin{corollary}
Let $K$ be either $\bQ$ or a quadratic field, 
$\fp$ a prime in $K$, 
$X=\Spec\cO_K\setminus\{\fp\}$ and 
$\cM\subseteq X$ a set of primes of density $1$. 
If $G:=\pi_1^{\mathrm{ab}}(X)$ is infinite, 
there is a natural exact sequence 
\[
1\to \prod_{a\in\cM} G/\overline{\langle s_a\rangle}
\to\Aut(Q^{\mathrm{ab}}(X, \cM))
\to \prod_{l\in \Spec\bZ} S_{c^{-1}(l)}. 
\]

Furthermore, 
$\Aut(Q^{\mathrm{ab}}(X, \cM))$ 
has a subgroup isomorphic to $\Aut(X, \cM)$, 
and except in the case where $K$ is real quadratic, we have 
\[
\Aut(Q^{\mathrm{ab}}(X, \cM)) = 
\Aut(X, \cM)\ltimes \prod_{a\in\cM} G/\overline{\langle s_a\rangle}. 
\]
\end{corollary}

\begin{proof}
By the main theorem, 
an automorphism $\varphi$ of $Q^{\mathrm{ab}}(X, \cM)$
commutes with $c\circ \pi$. 
In other words, the permutation on $\cM$ induced by $\varphi$ 
preserves fibers of $c$, 
hence the first assertion. 

By the previous proposition, 
a choice of section of $\pi$ 
gives a homomorphism 
$\Aut(X, \cM)\to\Aut(Q^{\mathrm{ab}}(X, \cM))$. 
The composite $\Aut(X, \cM)\to S_\cM$ 
corresponds to the natural action on $\cM$, 
so it is injective. 
By the main theorem, 
$\Aut(Q)$ is generated by $\Aut(X, \cM)$ 
and automorphisms preserving fibers of $\pi$ 
except in the real quadratic case, 
hence the second assertion. 
\end{proof}

\begin{example}
If $K$ is quadratic, 
$\prod_{l\in \Spec\bZ} S_{c^{-1}(l)}$ 
is much bigger than $\Aut(X, \cM)$. 
In the real quadratic case, 
it is possible that 
$\Aut(Q^{\mathrm{ab}}(X, \cM))\to \prod_{l\in \Spec\bZ} S_{c^{-1}(l)}$ 
is surjective. 

In fact, let $K=\bQ(\sqrt{5})$ and 
$\fp=3\cO_K$, which is a prime ideal, 
and $\cM$ the set of primes of $K$ different from $\fp$ and $(\sqrt{5})$. 

Let $\tilde{K}$ be the maximal abelian extension unramified outside $\fp$. 
It contains $K(\zeta_{3^\infty})$. 
We have $K\cap \bQ(\zeta_{3^\infty})=\bQ$ 
since $K$ is unramified over $3$. 
Therefore 
$\Gal(K(\zeta_{3^\infty})/K)\cong\Gal(\bQ(\zeta_{3^\infty})/\bQ)\cong\bZ_3^\times
\cong (\bZ/2\bZ)\times\bZ_3$. 

Since $K$ has narrow class number $1$, 
$\Gal(\tilde{K}/K)$ is isomorphic to $\cO_{K_\fp}^\times/\overline{\cO_K^{\times, +}}$. 
The group $\cO_{K_\fp}^\times$ can be decomposed as 
$\mu_8 \times (1+3\cO_K)$. 
There is a totally positive unit 
$\gamma:=\{(1+\sqrt{5})/2\}^2$, 
and it has order $4$ modulo $1+3\cO_K$. 
We have $\gamma^4=1+3\cdot\frac{15+7\sqrt{5}}{2}$, 
which spans a primitive $\bZ_3$-submodule of $1+3\cO_K$. 
Thus 
$\cO_{K_\fp}^\times/\overline{\cO_K^{\times, +}}$ is isomorphic to 
a quotient of $(\bZ/2\bZ)\times \bZ_3$. 
It follows that the surjection 
$\Gal(\tilde{K}/K)\to\Gal(K(\zeta_{3^\infty})/K)$ 
is in fact an isomorphism, 
i.e. $\tilde{K}=K(\zeta_{3^\infty})$. 

Let $l$ be a prime that splits as $\fl\cdot\bar{\fl}$ in $K$. 
Then $\fl$ and $\bar{\fl}$ define the same element of $G:=\Gal(\tilde{K}/K)$, 
i.e. the automorphism given by $\zeta_{3^n}\mapsto\zeta_{3^n}^l$. 
Thus we can exchange the fibers of $Q^{\mathrm{ab}}(X, \cM)$ 
over $\fl$ and $\bar{\fl}$ by a quandle automorphism. 
In this case we have 
$\Aut(Q^{\mathrm{ab}}(X, \cM))= 
\prod_{l\in \Spec\bZ} 
\left(S_{c^{-1}(l)} \ltimes 
\prod_{\fl\in c^{-1}(l)} \bZ_3^\times/\overline{\langle N(\fl) \rangle}\right)$. 
\end{example}

\end{document}